\documentclass{amsart}
\usepackage{latexsym,amssymb,scalerel}
\usepackage{caption,subcaption}

\newtheorem{theorem}{Theorem}[section]

\newtheorem{proposition}[theorem]{Proposition}
\newtheorem{corollary}[theorem]{Corollary}

\newtheorem{definition}[theorem]{Definition}

\newcommand{\power}[1]{{\scaleobj{0.8}{(#1)}}}
\newcommand{\mo}{\power{-1}}

\author{Ian Hawthorn}
\title{Sandwich Structures from Arbitrary Functions in Group Theory}
\thanks{Thanks to Dr. Tim Stokes for advice and for introducing me to prover9}
\subjclass[2010]{Primary: 20N99}
\begin{document}
\begin{abstract} 
Functions between groups with the property that all function conjugates 
are inverse preserving are called sandwich morphisms.  These maps preserve 
a structure within the group known as the sandwich structure. 
Sandwich structures are left distributive idempotent left involutary magmas.
These provide a generalisation of groups which we call a sandwich. This paper 
explores sandwiches and their relationship to groups.
\end{abstract}

\maketitle
\section{Introduction}

Group homomorphisms are maps between groups that preserve group structure.
In earlier papers (\cite{paper1},\cite{nilseries}) the author has looked 
at arbitrary functions between groups which are not in general homomorphisms,
but which in some sense can be said to partially preserve group structure.

The study of such functions is closely related to the study of generalisations
of groups. Given a generalisation of groups, the functions between groups 
which are morphisms for that generalisation will give an interesting collection
of not-quite-homomorphisms. Conversely given a collection of 
not-quite-homomorphisms it is natural to look at which properties of groups 
they preserve. A generalisation of groups may then be obtained by considering 
algebraic structures with those properties.

Not all interesting sets of arbitrary functions will give a group 
generalisation in this fashion. But when this can be done it may lead 
us to generalisations of groups that we might not otherwise have considered.

In this paper we take a fairly interesting class of functions between groups,
namely those for which all function conjugates are inverse preserving, and we 
use this set of functions to obtain a generalisation of groups.

\bigskip
\pagebreak[2]

\section{Function Conjugation and Inverse Preserving Functions}

A notion of function conjugation was introduced in~\cite{paper1}. 

If \( f : G \rightarrow H \) is an arbitrary function between 
finite groups and \( a \in G \), then we define a new function
\( f^a(x) = f(a)^{-1}f(ax) \) which we call \textbf{conjugate} of 
\( f \) by \( a \).

Clearly \( f \) is a group homomorphism if and only if \( f^a = f \) for 
all \( a \in G \). Consider for example the inverse function 
\( \mo : g \mapsto g^{-1} \). Then \( \mo^a(x) = ax^{-1}a^{-1} = [\mo(x)]^a \),
hence function conjugation generalises the usual conjugate. 

Note that \( f^a(1) = 1 \), hence conjugation maps the set of all functions 
onto the set of identity preserving ones. Furthermore function conjugation 
defines a group action of \( G \) on the set of identity preserving functions
mapping from \( G \) to \( H \) since \( f^1 = f \) and \( (f^a)^b = f^{ab} \).
Homomorphisms are precisely the functions invariant under this action.

\medskip
This function action was defined from the left. A similar action can be defined
from the right. We introduce temporary notation \( f^{|x>}(g) = f^x(g) \) for 
the action from the left and define action from the right by 
\( f^{<x|}(g) = f(gx^{-1})f(x^{-1})^{-1} \).
Then \( f^{<1|} = f \) and \( (f^{<a|})^{<b|} = f^{<ab|} \) so this 
is indeed an action.

We might expect that the actions from the left and right are related.
In fact they are equivalent. An intertwining map is given by
\( f^{\mo}(x) = f(x^{-1})^{-1} \). Note that 
\( {\left(f^\mo\right)}^\mo  = f \) hence \( f \mapsto f^\mo \) 
is of order two and in particular is a bijection. One can easily check that
\begin{flalign*}
(f^{<x|})^{\mo} &= (f^{\mo})^{|x>}  \\
(f^{|x>})^{\mo} &= (f^{\mo})^{<x|}
\end{flalign*}
so this defines an equivalence between the left and right actions.

As the two actions are equivalent it is reasonable to look at left actions
via our initial less cumbersome notation and use \( ((f^\mo)^x)^\mo \) to 
refer to the right action should this be necessary.

The left and right actions will be identical if for all \( x \in G \) we
have  \( (f^\mo)^x = (f^x)^\mo \), in other words if and only if function 
conjugation commutes with the inverse map.  

A function with the property that \( f^\mo = f \) is called 
\textbf{inverse preserving} as it preserves the relationship 
of being inverse. Inverse preserving functions are easy to construct. 
Furthermore any odd collection of functions which is closed under 
the map \( f \mapsto f^\mo \) must contain an inverse preserving function.

Note that if \( f \) is inverse preserving then in general \( f^a \) need 
not be inverse preserving. However it is the case that \( \mo^a \) is 
inverse preserving for all \( a \in G \).  If all function conjugates of 
\( f \) are inverse preserving then we say that \( f \) is 
\textbf{strongly inverse preserving}. Strongly inverse preserving functions 
were first introduced in~\cite{nilseries} although little is done with them in 
that paper. The property of being strongly inverse preserving is an interesting
and surprisingly strong constraint on \( f \) as we shall show in the next 
section.

If \( f \) is strongly inverse preserving then \( f^1 \) will be inverse 
preserving, however \( f \) itself need not be inverse preserving.
The left multiplication function \( x \mapsto ax \) where \( a \ne 1 \) 
is a strongly inverse preserving function which is not inverse preserving. 
It is usual to consider only identity preserving functions when working with
function actions in which case \( f^1 = f \) and hence strongly identity
preserving implies identity preserving in this case.

If \(f \) is identity preserving and strongly inverse preserving then 
the left and right conjugation actions on \( f \) are the same. Conversely 
an inverse preserving function for which left and right actions are the 
same will be strongly inverse preserving.

\bigskip
\pagebreak[2]

\section{Sandwich Morphisms and Structures}

We call a function \( f : G \rightarrow H \) between groups a 
\textbf{sandwich morphism} if \( f(ab^{-1}a) = f(a)f(b)^{-1}f(a) \) for all 
\( a,b \in G \). Initial examples are homomorphisms and the inverse function.
\begin{proposition}
A function \( f:G\rightarrow H \) between finite groups is strongly inverse 
preserving if and only if it is a sandwich morphism.
\end{proposition}
\begin{proof}
If \( f \) is strongly inverse preserving then 
\( f^a(x^{-1}) = \left(f^a(x)\right)^{-1} \) and hence 
\( f(a)^{-1}f(ax^{-1}) = \left(f(a)^{-1}f(ax)\right)^{-1} \) for all 
\( a,x\in G \). Rearranging we obtain
\begin{equation}
f(ax)f(a)^{-1}f(ax^{-1}) = f(a)f(ax^{-1})^{-1}f(a)
\end{equation}
and substituting \( x = b^{-1}a \) then gives the sandwich morphism property.

Conversely if \( f \) is a sandwich morphism then
\begin{flalign*}
f^a(b^{-1}) &= f(a)^{-1}f(ab^{-1}) \\
            &= f(a)^{-1}f(a(ab)^{-1}a) \\
            &= f(ab)^{-1}f(a) \\
            &= \left(f(a)^{-1}f(ab)\right)^{-1}= \left(f^a(b)\right)^{-1}
\end{flalign*}
and we can conclude that \( f \) is strongly inverse preserving as claimed.
\end{proof}

Sandwich morphisms are those functions which preserve the structure in the 
group specified by the binary operation \( (a,b) \mapsto a.b^{-1}.a \). 
We call this binary operation the \textbf{sandwich product} and the 
structure that it imposes on a group will be called the sandwich structure. 

We wish to study the sandwich product as a binary operation and to 
facilitate this we will denote it as simply \( ab \); using \( a.b \) 
to denote the usual group operation. Hence \( ab = a.b^{-1}.a \).

\begin{proposition}
\label{groupsandwich}
 The sandwich product in a finite group \( G \) has the 
following properties
\begin{description}
\item[Left Distributivity (LD)]
\( (ab)(ac) = a(bc) \) for all \( a,b,c \in G \).
\item[Idempotency (II)] \( aa = a \) for all \( a \in G \).
\item[Left Involutary (LI)] \( a(ab) = b \) for all \( a,b \in G \).
\item[Left Symmetry (LS)] If \( ab = b \) then \( ba = a \). 
\end{description}  
\end{proposition}
\begin{proof}
All these properties can be directly checked by expanding out 
in terms of the group product.
\end{proof}

Note that there is no mention of identities or inverses. Clearly in a 
group we have identities and inverses, but we cannot distinguish them from
other elements using only the sandwich product. The sandwich product does
relate them however. If we know the identity element \( 1 \) then we can
recover the inverses from the sandwich product by defining \( x^{-1} = 1x \).

The obvious next step is to throw away the group. This leads us to make 
the following definition.

\begin{definition}
A \textbf{sandwich} is a set \( G \) with a binary operation denoted
\( ab \in G \) which satisfies the properties in 
proposition~\ref{groupsandwich}.
A \textbf{subsandwich} \( T \le S \) of a sandwich \( S \) is
a subset which is itself a sandwich. This will be the case if and 
only if \( T \) is closed under the sandwich operation.
\end{definition}

In more technical terms a sandwich is a left involutary left distributive 
left symmetric idempotent magma. The properties as stated are independent 
in the sense that none can be proved from the others. Counterexamples of small
order which demonstrate this were generated using the program Mace4 by 
W.McCune~\cite{prover9}. These are described by matrices which specify 
Cayley tables for the binary operation in which elements are labelled 
\( 0 \), \( 1 \) , \( 2 \) etc and are listed in figure~\ref{ex1}.

\begin{figure}
\centering
\begin{subfigure}[b]{4em}
$ \left(\begin{array}{cc}
1&0\\1&0
\end{array}\right)$
\caption{Not II}
\end{subfigure}
\quad
\begin{subfigure}[b]{6em}
$\left(\begin{array}{ccc}
0&1&2\\2&1&0\\0&1&2
\end{array}\right)$
\caption{Not LS}
\end{subfigure}
\quad
\begin{subfigure}[b]{8em}
$\left(\begin{array}{cccc}
0&2&3&1\\3&1&0&2\\1&3&2&0\\2&0&1&3
\end{array}\right)$
\caption{Not LI}
\end{subfigure}
\quad
\begin{subfigure}[b]{10em}
$\left(\begin{array}{ccccc}
0&1&2&3&4\\0&1&3&2&4\\0&3&2&1&4\\1&0&4&3&2\\3&1&2&0&4
\end{array}\right)$
\caption{Not LD}
\end{subfigure}
\caption{Cayley tables for counterexamples. Each has all 
of the sandwich properties except the indicated one.}
\label{ex1}
\end{figure}

\begin{proposition}
The following identities hold for all \( a,b,c \) in a sandwich \( S \).
\begin{enumerate}
\item \( ab = ac \Rightarrow b = c \).
\item \( (ab)c = a(b(ac)) \) for all \( a,b,c \in S \).
\item \( (ab)c = bc \Leftrightarrow (ba)c = ac \) for all \( a,b,c \in S \).
\end{enumerate}
\end{proposition}
\begin{proof} These follow directly from the definition of a sandwich.
\begin{enumerate} 
\item \( ab = ac \Rightarrow a(ab) = a(ac) \Rightarrow b = c \)
\item \( (ab)c = (ab)(a(ac)) = a(b(ac)) \)
\item \( (ab)c = bc \Leftrightarrow a(b(ac)) = bc 
                    \Leftrightarrow b(ac)) = a(bc)
                    \Leftrightarrow ac = b(a(bc)) 
                    \Leftrightarrow ac = (ba)c \)
\end{enumerate}
\end{proof}
The first identity is the left cancellative (LC) property which we proved 
from the left involutary property ( LI $\Rightarrow$ LC ). This raises the 
question of whether the converse is also true for magmas satisfying the 
other sandwich properties. This is not the case as is demonstrated by the
non LI example in figure~\ref{ex1} which is LC.

Since sandwiches are left cancellative multiplication from the left
is transitive in a sandwich, which can be useful. Sandwiches are not 
usually right cancellative however. In particular sandwiches can contain
right zero elements. Indeed a sandwich can consist only of right zero 
elements as we now show.

A right zero semigroup is a set with the multiplication \( ab = b \).
These are sandwiches as one can directly check. We will call them 
\textbf{right zero sandwiches}. There is a unique right zero sandwich 
of every order. A right zero sandwich with at least 2 elements is not
right cancellative since \( ab = bb = b \).

The left distributive property tells us that multiplication from 
the left is a sandwich automorphism since \( a(xy) = (ax)(ay) \). Moreover 
it is an involution since \( a(ax) = x \). The automorphisms of a sandwich 
form a group just like the automorphisms of any algebraic object. Hence
we have a natural map from a sandwich to its automorphism group defined 
by mapping the element \( a \in S \) to the left multiplication function 
\( L_a : x \mapsto ax \). What are the properties of this natural map?

We have \( L(ab)(x) = (ab)x = (ab)(a(ax)) = a(b(ax)) = L(a)L(b)L(a)(x) \).
Furthermore since \( L(b) = L(b)^{-1} \) we can write 
\begin{equation}
L(ab) = L(a)L(b)^{-1}L(a)
\end{equation}
Hence this natural map is a sandwich homomorphism to the sandwich structure of 
the automorphism group. We will call a sandwich a \textbf{group sandwich} if
it is the sandwich structure of a group. And we will call a sandwich a 
\textbf{group subsandwich} if it is a subsandwich of a group sandwich. Hence 
left multiplication defines a sandwich homomorphism which maps an arbitrary 
sandwich onto a group subsandwich.

The congruences of this map on a sandwich \( S \) are equivalence classes 
under the relation \( a \sim b \) iff \( L(a) = L(b) \) which means 
\( ax = bx \) for all \( x \in S \). In particular this is true for 
\( x = b \) which gives \( ab = b \). 
 
Hence a congruence class is a subsandwich \( T \subseteq S \) where 
\( ab = b \) for all \( a,b \in T \). Thus the congruences for this map 
are right zero sandwiches.

We have proved
\begin{theorem}\label{clubofr0}
Every sandwich is a group subsandwich of right zero sandwiches.
\end{theorem}

As a consequence of this theorem any sandwich which has no non-trivial 
right zero subsandwiches must be a group subsandwich. Hence any condition 
on a sandwich that prevents it from having non-trivial right zero 
subsandwiches will result in it being a group subsandwich.

\begin{corollary}
A right cancellative sandwich is a group subsandwich.
\end{corollary}

What does this natural map look like if we apply it to a group?
\medskip

Let \( G \) be a group with sandwich product \( ab = a.b^{-1}.a \) and 
consider the congruence on \( G \) defined by the natural map into the 
sandwich automorphism group. If \( a \sim b \) then \( L(a) = L(b) \) and 
\( a.x^{-1}.a = b.x^{-1}.b \) for all \( x \in G \). 
 
Setting \( x = a \) or \( x = b \)  gives \( a.b^{-1} = b.a^{-1} \) and also
\( b^{-1}.a = a^{-1}.b \) . Let \( a.b^{-1} = e \) and \( a^{-1}.b = f \) so 
that \( e^2 = 1 \) and \( f^2 = 1 \). Then for all \( x \in G \) we must have 
\( x^{-1}.e = f.x^{-1} \) . Setting \( x = 1 \) gives \( e = f \) and hence
\( ex = xe \) for all \( x \in G \). It follows that \( e \in Z(G) \). 

Thus if \( a \) and \( b \) are equivalent then \( b = e.a \) for some 
element \( e \in Z(G) \) with \( e^2 = 1 \). Conversely if \( e \) is an 
element with these properties then \( a \) is equivalent to \( e.a \) for all 
\( a \in G \). This proves the following.

\begin{proposition} The congruences of the natural map from the sandwich
structure of a group into the sandwich structure of its sandwich automorphism
group are cosets of the subgroup consisting of all central elements of order
\( 1 \) or  \( 2 \).
\end{proposition}

\begin{corollary} If \( G \) has no central elements of order \( 2 \) then
the natural map into its sandwich automorphism group is 1-1. 
\end{corollary}

If the natural map is 1-1 then it gives a sandwich isomorphism onto its 
image. The image of this natural map consists of all elements of the form
\( L(a) \) in the sandwich automorphism group. However all these elements
have order \( 2 \) since \( L(a)L(a)(x) = a(ax) = x \).

This is very interesting. Effectively we have a sandwich isomorphism onto 
a subsandwich of a group, and all the elements of this subsandwich are 
group elements of order 1 or 2. Of course the subset of all elements of 
order $1$ or $2$ in a group need not constitute a subgroup. However it 
is a subsandwich.

\begin{proposition}
Let \( G \) be a group and let \( 2(G) = \{ x \in G : x^2 = 1 \} \). 
Then \( 2(G) \) is a subsandwich of \( G \).
\end{proposition}
\begin{proof}
It is enough to show that \( 2(G) \) is closed under the sandwich product. 
But if \( a^2 = 1 \) and \( b^2 = 1 \) then \( (a.b^{-1}.a)^2 \) = 1 so
this is true.
\end{proof}

The natural map takes a sandwich into this subsandwich of its sandwich 
automorphism group.
\medskip

We next consider the question of whether all sandwiches must arise from the 
sandwich structures of groups.

\begin{proposition}\label{r0club}
All right zero sandwiches are group subsandwiches.
\end{proposition}
\begin{proof}
Consider the right zero sandwich of order \( n \).  This is a set of \( n \) 
elements with the product \( ab = b \). There is only one such algebraic 
structure up to isomorphism.

Let \( G \) be an elementary abelian \( 2 \)-group of order \( 2^k \) 
with \( 2^k \ge n \). Then the sandwich product in \( G \) is 
\( ab = a.b^{-1}.a = a^2.b^{-1} = 1.b = b \) and thus the sandwich 
structure of \( G \) is a right zero sandwich of order \( 2^k \). 
Every subset of a right zero sandwich is a right zero subsandwich.
Since \( 2^k \ge n \) then \( G \) has a subset (and hence a right
zero subsandwich) of order \( n \). IT follows that the right zero 
sandwich of order \( n \) is a group subsandwich as claimed.
\end{proof}

\begin{proposition}
A right zero sandwich is a group sandwich if and only if it has
order \( 2^k \) for some \( k \).
\end{proposition}
\begin{proof} The sandwich structure of an elementary abelian \( 2 \)-group
is the right zero sandwich of order \( 2^k \). So this is a group sandwich.

Conversely let \( S \) be a group of order \( n \) with 
sandwich product \( ab = b \) for all \(a,b \in S \). Then 
\( a1 = a.1^{-1}.a = 1 \) and so all elements of the group \( S \)
have order \( 2 \). It follows that \( S \) is an elementary abelian 
\( 2 \)-group and must therefore have order \( n = 2^k \) for some \( k \).
\end{proof}

\begin{corollary} Not all sandwiches are group sandwiches.
\end{corollary}
\begin{proof} A right zero semigroup of order three provides a counterexample.
\end{proof}

The question of whether all sandwiches are group subsandwiches remains 
open.  Theorem~\ref{clubofr0} and proposition~\ref{r0club} suggest that
if this is not true then it may not be easy to construct counterexamples.


\begin{thebibliography}{9}
\bibitem{paper1}{Ian Hawthorn and Yue Guo, 
		\textit{Arbitrary Functions in Group Theory},
		New Zealand Journal of Mathematics, Vol 45 (2015),1-9}
\bibitem{Robinson}{Derek J. S. Robinson, 
		\textit{A course in the theory of groups},
		2nd edition, Springer 1996}
\bibitem{nilseries}{Ian Hawthorn, 
	\textit{Nil series from Arbitrary Functions in Group Theory},
	accepted for publication in CMUC}
\bibitem{prover9}{W.McCune, "Prover9 and Mace4",
 http://www.cs.unm.edu/~mccune/Prover9, 2005-2010.}
\end{thebibliography}
\end{document}